\documentclass{amsart}
\usepackage[latin1]{inputenc}
\usepackage[active]{srcltx}

\usepackage{amsfonts,enumerate}

\usepackage[mathscr]{euscript}
\usepackage{epsfig}
\usepackage{latexsym}

\usepackage[all]{xy}

\usepackage{amssymb}
\usepackage{amsthm}
\usepackage{amsmath}
\usepackage{amsxtra}

.tfm
.tfm

\theoremstyle{plain}
\newtheorem{prop}{Proposition}[section]
\newtheorem{thm}[prop]{Theorem}

\newtheorem{lemma}[prop]{Lemma}
\newtheorem{conj}{Conjecture}

\newtheorem{ass}{Assumption}
\newtheorem*{thmA}{Theorem}

\theoremstyle{definition}
\newtheorem{defi}[prop]{Definition}

\theoremstyle{remark}

\newtheorem{remark}{Remark}

\numberwithin{table}{section}

\DeclareMathOperator{\Frob}{Frob}
\DeclareMathOperator{\Symm}{Symm}
\DeclareMathOperator{\Spec}{Spec}
\DeclareMathOperator{\HT}{HT}
\DeclareMathOperator{\WD}{WD}
\DeclareMathOperator{\Gal}{Gal}

\DeclareMathOperator{\Ind}{Ind}

\DeclareMathOperator{\im}{Im}

\newcommand{\R}{\mathscr R}

\newcommand{\F}{\mathbb F}

\newcommand{\N}{\mathcal{N}}

\newcommand{\Om}{{\mathscr{O}}}

\newcommand{\GL}{{\rm GL}}
\newcommand{\SL}{{\rm SL}}
\newcommand{\PGL}{{\rm PGL}}
\newcommand{\PSL}{{\rm PSL}}

\def\TT{\mathbb T}
\def\ZZ{\mathbb Z}

\def\FF{\mathbb F}
\def\QQ{\mathbb Q}

\def\CC{\mathbb C}

\def\<#1>{{\left\langle{#1}\right\rangle}}

\def\Q{{\mathbb Q}}             

\def\id#1{{\mathfrak{#1}}}      
\def\normid#1{{\norm{\id{#1}}}}
\DeclareMathOperator{\norm}{{\mathscr N}}

\begin{document}

\title[Connectedness of Hecke Algebras]{Connectedness of Hecke Algebras and the Rayuela conjecture:
a path to functoriality and modularity}

\author{Luis Dieulefait}
\address{Departament d'\`Algebra i Geometria, Facultat de Matem\`atiques,
  Universitat de \break Barcelona, Gran Via de les Corts Catalanes,
  585. 08007 Barcelona}
\email{ldieulefait@ub.edu}
\thanks{L.D. partially supported by MICINN grants MTM2012-33830 and by an ICREA Academia Research Prize}

\author{Ariel Pacetti}
\address{Departamento de Matem\'atica, Facultad de Ciencias Exactas y Naturales, Universidad de Buenos Aires and IMAS, CONICET, Argentina}
\email{apacetti@dm.uba.ar}
\thanks{AP was partially supported by CONICET PIP 2010-2012 GI and FonCyT BID-PICT 2010-0681.}
\date{August 26, 2014}
\dedicatory{A la memoria de Julio Cortazar, autor de "Rayuela", al cumplirse hoy cien años de su nacimiento.}

\begin{abstract}
  Let $\rho_1$ and $\rho_2$ be a pair of residual, odd, absolutely
  irreducible two-dimensional Galois representations of a totally real
  number field $F$.  In this article we propose a conjecture asserting
  existence of ``safe" chains of compatible systems of Galois
  representations linking $\rho_1$ to $\rho_2$. Such conjecture
  implies the generalized Serre's conjecture and is equivalent to
  Serre's conjecture under a modular version of it. We prove a weak
  version of the modular variant using the connectedness of
  certain Hecke algebras, and we comment on possible applications of
  these results to establish some cases of Langlands functoriality.
\end{abstract}

\maketitle


\section{Introduction}

Let $F$ be a totally real number field. In \cite{KK} Khare and Kiming
conjectured that given $\rho_1$ and $\rho_2$ two odd, absolutely
irreducible, two-dimensional residual representations of the absolute
Galois group of $F$ with values on finite fields of different prime
characteristics $\ell$ and $\ell'$, satisfying certain local
compatibilities (for all primes $q$ different from $\ell$ and $\ell'$,
there exists a Weil-Deligne parameter and a choice of integral model
such that their reductions modulo $\ell$ and $\ell'$ are isomorphic to
the restrictions of $\rho_1$ and $\rho_2$ to a decomposition group at
$q$), there exists a modular form which residually coincides
with $\rho_1$ at $\ell$ and $\rho_2$ at $\ell'$. This is a very strong
statement and no evidence for its truth is known. In particular it
implies that given two Hilbert modular forms $f$ and $g$ over $F$, and
two prime numbers $\ell$ and $\ell'$, if we assume suitable local
compatibilities, then there exists a Hilbert modular form $h$ which is
congruent to $f$ modulo $\ell$ and to $g$ modulo $\ell'$.

The purpose of this article is to state a weaker conjecture called
{\bf The Rayuela Conjecture} involving ``chains of congruent
systems" connecting two given Galois representations. The idea is
not to link any given pair of residual Galois representations via a
single compatible system of Galois representations but via a chain
of such systems, meaning that each of them has to be congruent to
the next one in the chain, modulo a suitable prime, and it will also
be required that some reasonable properties are being satisfied at
each of these congruences (this will be done in
Section~\ref{sec:rayuela}).
Such conjecture has important consequences, like Serre's conjectures
for totally real number fields or base change for totally real
number fields (the second one follows from a weaker version that we
will explain later). The conjecture for the field of rational
numbers follows from the results proved by Dieulefait in
\cite{LuisBC2}  (the reader can easily check that the method used to
prove base change in loc. cit. implies, in fact is based on, the
fact that this conjecture is true over $\Q$ at least for modular
Galois representations, and this combined with the truth of Serre's
conjecture over $\Q$ implies the conjecture).

Although we are not able to prove the conjecture itself, in Sections
~\ref{sec:abstract} and ~\ref{sec:killlevel} we prove that some
(weak) variant of it holds in the modular world. This is a
generalization of a Theorem of Mazur, a result that he proved for
the case of
weight $2$ modular forms of prime level $N$.

In this article, for general weights and levels, and for any totally real number
field, we extend Mazur's result, and the control we get in the level
and weights of the modular forms involved might be used to prove the
modular version of the conjecture (although there are still some
ingredients missing to get the full statement).

Also we will show that our Conjecture is equivalent to the generalized
Serre's conjecture proposed in \cite{BDJ}, at least if some
strengthening of the available Modularity Lifting Theorems is
assumed. What we show is how one can manipulate a pair of modular
Galois representations to end up in a controlled situation where both
representations have the same Hodge-Tate weights and ramification in a
small controlled set of primes (and then this is combined with Mazur's
result).

Combining the ideas of the present article with some nowadays
standard tricks (like the use of Micro Good Dihedral primes), one
can prove for some small real quadratic fields a Base Change
theorem, all that is left is a finite computational check (to ensure
that some Hecke Algebra of known level and weight is connected in a
good way). This is part of a work in progress of the authors but
more will be said at the end of the present article.

For most of the chains constructed in this paper, the construction
carries on for the case of abstract Galois representations (the reader
should keep in mind that the possibility of constructing congruences
between abstract Galois representations where some local information
changes and the existence of compatible systems containing most
geometric $p$-adic Galois representations are the two main technical
ingredients in the proofs of Serre's conjecture over $\Q$), and then
again the use of Micro Good Dihedral primes combined with the results
in this paper is enough to reduce the proof of Serre's conjecture over
a given small real quadratic field to some special cases of Galois
representations with small invariants (Serre's weights and
conductor). Thus it is perfectly conceivable that one can give a
complete proof of Serre's conjecture over a small real quadratic field
$F$, by just completing this process with a few extra steps designed
to remove the Micro Good Dihedral prime from the level (eventually
relying on results of Skinner and Wiles~\cite{SW} if the residually
reducible case has to be considered) and end up in some ``base case
for modularity" over $F$, such as the modularity/non-existence of
semistable abelian varieties of conductor $3$ over $\Q(\sqrt{5})$
proved by Schoof (\cite{Schoof}). We plan to check over which real
quadratic fields, such proof of Serre's conjecture can be completed in
a future work.

Section~\ref{sec:MLT} contains the Modularity Lifting Theorems (that
will be denoted MLT) used and needed in the present article. We have
included a mixed case that has only been proved in weight $2$
situations, yet we assume its truth in more general situations. We
believe that such a result can be deduced from the techniques in
 \cite{BLGGT2}, but we haven't formally checked this. \\


\noindent{\bf Conventions and notations:} If $F$ is a number field,
we denote by $\Om_F$ its ring of integers. By $G_F$ we denote the
Galois group $\Gal(\bar{F}/F)$. All Galois representations (residual
or $\ell$-adic) are assumed to be continuous.  Let $\id{p}$ be a
prime ideal in $\Om_F$. We denote by $\Frob_{\id{p}}$ a Frobenius
element over $\id{p}$. \\

To ease notation, instead of specifying a prime $\pi$ in the field
of coefficients of a Galois representation and reducing mod $\pi$,
sometimes
 we will simply say that we reduce ``mod $p$", where
$p$ is the rational prime below $\pi$. We expect that this will be
no cause of
confusion.

Concerning local types of strictly compatible systems of Galois
representations, we will use the words ``Steinberg'', ``principal
series'' and ``supercuspidal'', to denote the local representations
they correspond to under local Langlands.

\bigskip

\noindent{\bf Acknowledgments:} we want to thank Fred Diamond for
useful comments and remarks and for the proof of
Lemma~\ref{lemma:Diamond}.

\section{The Rayuela Conjecture}\label{sec:rayuela}

There are nowadays many MLT that can be applied to propagate
modularity through a congruence between two $\ell$-adic Galois
representations. For us, the main interest in congruences between
Galois representations is in the case where a MLT holds in both
directions (i.e. modularity of either of the two representations
implies modularity of the other one), so we will call a congruence an
MLT congruence if the hypothesis of some MLT theorem are fulfilled in
both directions.  Since MLT hypothesis are becoming less restrictive
with time, some of the proofs we give make use of tricks that are
required to reduce to situations in which some of the available MLT
applies, but some of these tricks are very likely to become obsolete
in the near future (when new MLT are proved). There are also steps in
the chains that we are going to build in our attempt to connect two
given modular or abstract Galois representations that involve
congruences that are not known to be MLT. This is the unique reason
why we are not able to prove any strong result in this paper, such as
relative base change (see the discussion at the end of the
article).


Let us now state our conjecture for abstract Galois representations.

\begin{conj} {\bf The Rayuela Conjecture}:
Let $F$ be a totally real number field, $\ell_0$, $\ell_\infty$
prime numbers and
\[
\rho_i:\Gal(\bar{F}/F) \to \GL_2(\overline{\FF}_{\ell_i}),
\qquad{i=0, \infty},
\]
two absolutely irreducible odd Galois representations. Then there
exists a family of odd, absolutely irreducible, 2-dimensional strictly
compatible systems of Galois representations
$\{\rho_{i,\lambda}\}_{i=1}^n$ such that:
\begin{itemize}
\item $ \rho_0 \equiv \rho_{1,\lambda_0} \pmod{\lambda_0}$, with
$\lambda_0 \mid \ell_0$.
\item $ \rho_\infty \equiv \rho_{n,\lambda_\infty}\pmod{\lambda_\infty}$, with
$\lambda_\infty \mid \ell_\infty$.
\item for $i=1,\dots,n-1$ there exist $\lambda_i$ prime such that
$\rho_{i,\lambda_i} \equiv \rho_{i+1,\lambda_i} \pmod{\lambda_i}$,
\item all the congruences involved are MLT.
\end{itemize}
\label{conj:p1pn}
\end{conj}

\begin{remark}
The primes involved always are taken to be in the field of
coefficient of the corresponding system of representations (for the
first and second condition) or in the compositum of the two relevant
such fields (for the third conditions). From now on, this remark
applies to all congruence between compatible systems appearing in this paper.
\end{remark}

The last hypothesis of the conjecture implies that modularity of one
of the given  representations can be propagated through the chain of
congruences produced by the conjecture allowing to prove modularity
of the other one. It is thus clear that Conjecture~\ref{conj:p1pn}
implies Serre's generalized conjecture over $F$ by just taking
$\rho_\infty$ to be an irreducible residual representation attached
to any cuspidal Hilbert modular form over $F$.

At this time we think Conjecture~\ref{conj:p1pn} is out of reach,
but since it does not involve directly modular forms, it might lead
to a different attack to Serre's conjecture. Note that in the
hypothesis of the Conjecture, one can put $\rho_0$ inside a strictly
compatible system of Galois representations as done by
\cite{LuisFamilies} for representations over $\QQ$ and by
\cite{Snowden} for totally real fields (using Theorem 7.6.1 and
Corollary 1.1.2). The main difficulty for abstract representations
is to connect them. During this work we will show how to manipulate
abstract representations such that if we start with any pair of them
we can connect both, through suitable chains of MLT congruences,  to
representations having common values for their Serre's level and
weights, but we
cannot go any further so far. \\

For different purposes, such as applications to Langlands
functoriality, it is interesting to study the previous conjecture in
the modular setting. In this case we are able to prove part of it,
namely, we are able to build the chain of congruences but we can not
ensure that congruences are MLT at all steps. Still, we find this
interesting because advances in MLT theorems may eventually lead to
a proof of this modular variant of the conjecture, and this will be
evidence for the truth of the Rayuela conjecture. In fact it will
give a proof that this conjecture is equivalent to the generalized
Serre's conjecture over $F$ (the other implication being trivial as
already remarked). The modular variant is the following:

\begin{conj} Let $F$ be a totally real number field, $\ell_0$, $\ell_\infty$
prime numbers and
\[
\rho_i:\Gal(\bar{F}/F) \to \GL_2(\overline{\FF}_{\ell_i}),
\qquad{i=0, \infty},
\]
two absolutely irreducible odd Galois representations attached to
Hilbert modular newforms $f_0$ and $f_\infty$, respectively. Then
there exists a family of odd, absolutely irreducible, 2-dimensional
strictly compatible systems of Galois representations
$\{\rho_{i,\lambda}\}_{i=1}^n$ such that:
\begin{itemize}
\item $ \rho_0 \equiv \rho_{1,\lambda_0} \pmod{\lambda_0}$, with
$\lambda_0 \mid \ell_0$.
\item $ \rho_\infty \equiv \rho_{n,\lambda_\infty}\pmod{\lambda_\infty}$, with
$\lambda_\infty \mid \ell_\infty$.
\item for $i=1,\dots,n-1$ there exist $\lambda_i$ prime such that
$\rho_{i,\lambda_i} \equiv \rho_{i+1,\lambda_i} \pmod{\lambda_i}$,
\item all the congruences involved are MLT.
\end{itemize}

\label{conj:p1pnmod}
\end{conj}

The aim of this article is to prove part of
Conjecture~\ref{conj:p1pnmod} and show some implications of it related
to modularity and functoriality. It is clear that all the
representations appearing in Conjecture~\ref{conj:p1pnmod} are
modular. Clearly we have:

\begin{thm}
Serre's generalized conjecture $+$ Conjecture~\ref{conj:p1pnmod} imply
Conjecture~\ref{conj:p1pn}.
\end{thm}
\begin{proof}
This is clear, since if $\rho_0$ and $\rho_\infty$ are any two
Galois representation and if Serre's conjecture holds, they are
modular, and then they are connected in the right way by Conjecture
2.
\end{proof}

Conjecture~\ref{conj:p1pnmod} can be proved using standard arguments
mainly due to Mazur, if we remove the condition that the congruences
are MLT.


\begin{thmA}[Mazur]
Let $F$ be a totally real number field and $f_0, f_\infty$ two Hilbert
modular forms, whose weights are congruent modulo $2$. Then there
exists Hilbert modular forms $\{h_i\}_{i=1}^n$ such that $h_1 =
f_0$, $h_n = f_\infty$ and
such that
for $i=1,\dots,n-1$ there exist $\lambda_i$ prime such that $\rho_{h_i,\lambda_i} \equiv \rho_{h_{i+1},\lambda_i} \pmod{\lambda_i}$.
\label{thm:Mazur}
\end{thmA}

\begin{remark}
  Although connectedness of the Hecke algebra was proved by Mazur for
  classical modular forms of weight $2$ and prime level $N$ using the
  curve $X_0(N)$, the proof we will give is strongly based in
  his argument.
\end{remark}

\begin{remark}
  If modularity would propagate via any congruence, the previous
  Theorem would be equivalent to
  Conjecture~\ref{conj:p1pnmod}. Unfortunately this is not clear
  even for classical modular forms.
\end{remark}

The main result of the present article is the following:

\begin{thm}
Let $F$ be a totally real number field, $\ell_f$, $\ell_g$
prime numbers and
\[
\rho_f:\Gal(\bar{F}/F) \to \GL_2(\overline{\FF}_{\ell_f}),
\qquad \rho_g:\Gal(\bar{F}/F) \to \GL_2(\overline{\FF}_{\ell_g}),
\]
two absolutely irreducible odd Galois representations attached to
Hilbert modular newforms $f$ and $g$, respectively. Then there exists
prime ideals $\id{q}\subset \Om_F$ and $p \in \Q$ which splits
completely in $F$, Hilbert modular forms $f_\infty$ and $g_\infty$ of
parallel weight $2$ and level $\Gamma_0(p\id{q}^2)$ and two families of
odd, absolutely irreducible, 2-dimensional strictly compatible systems
of Galois representations $\{\rho_{i,\lambda}^f\}_{i=1}^{n_f}$,
$\{\rho_{i,\lambda}^g\}_{i=1}^{n_g}$, such that:
\begin{itemize}
\item $ \rho_f \equiv \rho_{1,\lambda_f}^f \pmod{\lambda_f}$, with
$\lambda_f \mid \ell_f$.
\item $ \rho_{f_\infty,\lambda_\infty} \equiv \rho_{n_f,\lambda_\infty}^f\pmod{\lambda_\infty}$, for some prime ideal $\lambda_\infty$.
\item for $i=1,\dots,n-1$ there exist $\lambda_i$ prime such that
$\rho_{i,\lambda_i}^f \equiv \rho_{i+1,\lambda_i}^f \pmod{\lambda_i}$,
\item $ \rho_g \equiv \rho_{1,\lambda_g}^g \pmod{\lambda_g}$, with
$\lambda_g \mid \ell_g$.
\item $ \rho_{g_\infty,\lambda_\infty} \equiv \rho_{n_g,\lambda_\infty}^g\pmod{\lambda_\infty}$, for some prime ideal $\lambda_\infty$.
\item for $i=1,\dots,n-1$ there exist $\lambda_i$ prime such that
$\rho_{i,\lambda_i}^g \equiv \rho_{i+1,\lambda_i}^g \pmod{\lambda_i}$,
\item all the congruences involved are MLT.
\end{itemize}
\label{thm:main}
\end{thm}

\begin{remark}
  The importance of our main result is that we can translate
  Conjecture~\ref{conj:p1pnmod} to a situation where the level and the
  weight are known. Some of the primes involved are not explicit, since they
  come from some application of Tchebotarev density theorem. See Section \ref{sec:last}.
\end{remark}


\section{MLT theorems used}
\label{sec:MLT}
In this section we just enumerate the MLT theorems that will be used
during this work. By $E$ we denote a finite extension of $\Q_\ell$.

\begin{thm}[MLT1]
Let $F$ be a totally real number field, and $\ell \ge 5$ a prime number
which splits completely in $F$. Let $\rho:G_F \to \GL_2(E)$ be a
continuous irreducible representation such that:
\begin{itemize}
\item $\rho$ ramifies only at finitely many primes,
\item $\overline{\rho}$ is odd,
\item $\rho|_{G_{F_v}}$ is potentially semi-stable for any $v \mid \ell$ with different
  Hodge-Tate weights.
\item The restriction $\overline{\rho}|_{G_{F(\xi_\ell)}}$ is absolutely irreducible.
\item $\bar{\rho} \sim \bar{\rho_f}$ for a Hilbert modular form $f$.
\end{itemize}
 Then $\rho$ is automorphic.
\label{thm:MLT1}
\end{thm}

\begin{proof}
This is Theorem 6.4 of \cite{Hu-Tan}.
\end{proof}

\begin{thm}[MLT2 - ordinary case] Let $F$ be a totally real field, $\ell$
  and odd prime and $\rho:G_F \to \GL_2(E)$ be a continuous
  irreducible representation such that:
  \begin{itemize}
  \item $\rho$ is unramified at all but finitely many primes.
  \item $\rho$ is de Rham at all primes above $\ell$.
  \item The reduction $\bar{\rho}$ is irreducible and
    $\bar{\rho}(G_{F(\xi_\ell)})\subset \GL_2(\overline{\FF_\ell})$ is adequate.
  \item $\rho$ is ordinary at all primes above $\ell$.
  \item $\bar{\rho}$ is ordinarily automorphic.
  \end{itemize}
  Then $\rho$ is ordinarily automorphic. If $\rho$ is also crystalline
  (resp. potentially crystalline), then $\rho$ is ordinarily
  automorphic of level prime to $\ell$ (resp. potentially level prime to
  $\ell$).
\label{thm:MLT2}
\end{thm}

\begin{proof}
  This is just Theorem $2.4.1$ of \cite{BLGGT2}.
\end{proof}

\begin{thm}[MLT3 - pot. diagonalizable case] Let $F$ be a totally real field, $\ell \ge 7$
be  and odd prime and $\rho:G_F \to \GL_2(E)$ be a continuous
  irreducible representation such that:
  \begin{itemize}
  \item $\rho$ is unramified at all but finitely many primes.
  \item $\rho$ is de Rham at all primes above $\ell$, with different Hodge-Tate numbers.
  \item $\rho|_{G_{F_{\lambda}}}$ is potentially diagonalizable for all $\lambda \mid \ell$.
  \item The restriction $\bar{\rho}|_{G_{F(\xi_\ell)}}$ is irreducible.
  \item $\bar{\rho}$ is either ordinarily automorphic or potentially diagonalizable automorphic.
  \end{itemize}
  Then $\rho$ is potentially diagonalizable automorphic (of level potentially prime to $\ell$).
\label{thm:MLT3}
\end{thm}

\begin{proof}
  This is Theorem $4.2.1$ of \cite{BLGGT2}. Note that although it is
  stated only for CM fields, one can chose a suitable CM extension and
  get the same result for totally real fields using solvable base change.
\end{proof}

\begin{remark}
  The hardest condition to check in the previous Theorem is that of
  $\rho$ being potentially diagonalizable, but this is satisfied if,
  in particular, $\ell$ is unramified in $F$, $\rho$ is crystalline at
  all primes above $\ell$ and the Hodge-Tate weights are in the
  Fontaine-Laffaille interval.
\end{remark}

 \begin{remark}
   The last result is stated for $\ell \ge 7$, BUT if $\ell$ is odd and $\overline{\rho}|G_{F(\xi_\ell)}$ is adequate, the same holds (see \cite{DiGe}).
 \end{remark}

 We also need a mixed variant. Recall the following property.

\begin{lemma}
  Let $F$ be a totally real field, and $\{\rho_{\lambda}\}$ be a
  strictly compatible system of continuous, odd, irreducible, parallel
  weight $2$ representations, then $\rho_{\lambda}$ is potentially
  Barsotti-Tate or ordinary at $\lambda$.
\label{lemma:typesatp}
\end{lemma}

\begin{proof}
  If the restriction to $\lambda$ is not potentially crystalline, then
  it is potentially semistable non-crystalline, in which case
  ordinariness is known. In fact, using potential modularity and the
  semistability at $\lambda$, the system is known to correspond to an
  abelian variety with potentially semistable reduction at $\lambda$,
  and potentially (over a suitable extension) to a parallel weight $2$
  Hilbert modular form which is Steinberg at $\lambda$.
\end{proof}

The following variant of the current MLT is not yet known:

\begin{ass}[MLT4 - mixed case] Let $F$ be a totally real field, $\ell$
  be and odd prime and $\rho:G_F \to \GL_2(E)$ be a continuous
  irreducible representation such that:
  \begin{itemize}
  \item $\rho$ is unramified at all but finitely many primes.
  \item $\rho$ is de Rham at all primes above $\ell$, with different Hodge-Tate numbers.
  \item $\rho|G_{F_{\lambda}}$ is potentially diagonalizable for some $\lambda \mid \ell$ and is ordinary at the others.
  \item The restriction $\bar{\rho}|_{G_{F(\xi_\ell)}}$ is irreducible and adequate.
  \item $\bar{\rho}$ is either ordinarily automorphic or potentially diagonalizably automorphic at the same places as $\bar{\rho}$.
  \end{itemize}
  Then $\rho$ is potentially automorphic (of level potentially prime to $\ell$).
\label{ass:MLT4}
\end{ass}

\begin{remark}
  We believe that using the tools developed in \cite{BLGGT2} this
  result is accessible. Not only we consider the above result
  accessible, but also if we restrict to the case where both Galois
  representations involved are of parallel weight $2$ (thus, because
  of the previous Lemma, locally potentially Barsotti-Tate or ordinary
  at all primes above $p$), it is a Theorem as proved in \cite{BD},
  Theorem 3.2.2.
\end{remark}

\begin{remark}
  We will need also need Assumption~\ref{ass:MLT4} to hold for $p=2$
  and parallel weight $2$.
\end{remark}

\begin{remark}
  To apply the previous MLT's we need
  $\overline{\rho}(G_{F(\xi_\ell)})$ to be irreducible and to be
  adequate. 
  When $\ell \ge 7$, adequacy is equivalent to irreducibility, so this
  imposes no extra hypothesis. Nevertheless, when $\ell = 3$ and $\ell
  = 5$ this is not the case. Theorem 1.5 of \cite{Gur} implies that if
  $\ell \ge 3$ and $\SL_2(\FF_{p^r}) \subset \im(\overline{\rho})$ for
  $r>1$ then the residual image of the representation (and the same for the image of
   its restriction to any abelian extension)
   is adequate as well. In what follows, we will only apply any of the previous MLT's in characteristics $3$ or $5$
AFTER having added a good dihedral
  prime (which is added modulo a prime greater than $5$)
   and from this it follows that this result of Guralnick can be
   applied ensuring adequacy
  (see \cite{LuisBC2} for more details).

\end{remark}

\section{Abstract representations:}\label{sec:abstract}

As mentioned in the introduction, most of the level/weight
manipulations that we are going to perform, work not only for modular
representation but for abstract ones as well. In this section we
will work in the greatest generality possible, and in the next
section we will restrict to modular representations putting emphasis
on the results that are nowadays only known for modular
representations. The results in this section are enough to prove
Mazur connectedness Theorem. We begin by recalling the following
well-known definition (see \cite{serre-book} for example). If
$\lambda_i$ is a prime in a number field $K$, we denote by $\ell_i$
the rational prime below $\lambda_i$ and by $L_{\ell_i}$ the set of primes in
$K$ dividing $\ell_i$ (which clearly contains $\lambda_i$).

\begin{defi}
  A \emph{compatible system} of Galois representations over $F$
  is a family of continuous Galois representations
\[
\rho_\lambda : \Gal(\overline{F}/F) \to \GL_2(K_\lambda),
\]
where $K$ is a finite extension of $\Q$ and $\lambda$ runs through
the prime ideals of $\Om_K$, which satisfy:
\begin{enumerate}
\item There exists a finite set of primes $S$ (independent of
  $\lambda$) such that $\rho_\lambda$ is unramified outside $S \cup
  L_\ell$.
\item For each pair of prime ideals $(\lambda_1,\lambda_2)$ in $\Om_K$
  and for each prime ideal $\id{p} \not \in S \cup
  L_{\ell_1} \cup L_{\ell_2}$, the characteristic polynomials
  $Q_{\id{p}}(x)$ of $\rho_{\lambda_i}(\Frob_{\id{p}})$ lie in $K[x]$
  and are equal.
\end{enumerate}
\end{defi}

The most important examples of such families are the ones arising
from the \'etale cohomology of a variety defined over $F$. In this
case we also have some control on the roots of the characteristic
polynomials, and some control on the $\lambda$-adic representation
at primes in $L$. This motivate the following definition (see
\cite{BLGGT2}, Section 5).

\begin{defi}
  A rank $2$ \emph{strictly compatible system} of Galois
  representations $\R$ of $G_F$ defined over $K$ is a $5$-tuple
\[
\R=(K,S,\{Q_{\id{p}}(x)\},\{\rho_\lambda\},\{H_\tau\}),
\]
where
\begin{enumerate}
\item $K$ is a number field.
\item $S$ is a finite set of primes of $F$.
\item for each prime $\id{p} \not \in S$, $Q_{\id{p}}(x)$ is a degree
  $2$ polynomial in $K[x]$.
\item For each prime $\lambda$ of $K$, the representation
\[
\rho_\lambda:G_F \to \GL_2(K_\lambda),
\]
is a continuous semi-simple representation such that:
\begin{itemize}
\item If $\id{p} \not \in S$ and $\id{p} \nmid \ell$, then
  $\rho_\lambda$ is unramified at $\id{p}$ and
  $\rho_\lambda(\Frob_{\id{p}})$ has characteristic polynomial
  $Q_{\id{p}}(x)$.
\item If $\id{p} \mid \ell$, then $\rho|_{G_{F_{\id{p}}}}$ is de Rham and in the case $\id{p} \notin S$, crystalline.
\end{itemize}
\item for $\tau:F \hookrightarrow \overline{K}$, $H_\tau$  contains $2$
  different integers such that for any $\overline{K} \hookrightarrow
  \overline{K}_\lambda$ over $K$, we have that
  $\HT_\tau(\rho_\lambda)=H_\tau$.
\item For each finite place $\id{p}$ of $F$ there exists a Weil-Deligne
  representation $\WD_{\id{p}}(\R)$ of $W_{F_{\id{p}}}$ over $\overline{K}$
  such that for each place $\lambda$ of $K$ not dividing the residue
  characteristic of $\id{p}$ and every $K$-linear embedding
  $\iota:\overline{K} \hookrightarrow \overline{K}_\lambda$, the push forward
  $\iota\WD_{\id{p}}(\R) \simeq
  \WD(\rho_\lambda|_{G_{F_{\id{p}}}})^{K\text{-ss}}$.
\end{enumerate}
\end{defi}

\begin{remark}
  If one starts with a $2$-dimensional continuous, odd, Galois
  representations over a totally real number field $F$, under some
  minor hypothesis (which are exactly the hypothesis for an MLT
  theorem to hold), one can prove that such representations is
  potentially modular. In particular, this implies that the
  representations is part of a strictly compatible system. Since all
  the congruences we will work with are where an MLT theorem works in
  both directions, without loss of generality, we will assume that all
  the representations come in strictly compatible systems.
\end{remark}

\begin{defi}
  Let $\{\rho_\lambda\}$ be a strictly compatible system of Galois
  representations. We say that the system is \emph{dihedral} if the
  images are compatible dihedral groups, i.e. if  there exists a
  quadratic extension $L/F$ (independent of $\lambda$) such that
  $\rho_\lambda$ is induced from a $\lambda$-adic character of $L$.
\end{defi}

\begin{lemma}
  The family $\{\rho_\lambda\}$ is dihedral if and only if one
  representation is dihedral.
\end{lemma}

\begin{proof}
  It is clear that if the whole family is dihedral, in particular any
  of them is dihedral. For the converse, let $\lambda$ be a prime
  ideal of $\Om_F$, and suppose that $\rho_{\lambda_0}$ is
  dihedral. Then there exists a quadratic extension $L/F$ and a
  $\lambda$-adic character $\chi_{\lambda_0}:\Gal(\overline{L}/L) \to
  K_{\lambda_0}$ such that $\rho_{\lambda_0} = \Ind_{G_L}^{G_F}
  \chi_{\lambda_0}$. We know that $\chi_{\lambda_0}$ is part of a
  strictly compatible system of $1$-dimensional representations
  $\{\chi_{\lambda}\}$, so we are led to prove that $\rho_{\lambda}
  \simeq \Ind_{G_L}^{G_F} \chi_{\lambda}$ for all primes $\lambda
  \subset \Om_F$. This comes from a straightforward computation, since the
  values of the traces (an even the whole characteristic polynomial)
  of $\rho_{\lambda_0}(\Frob_{\id{p}})$ are given in terms of the
  values of $\chi_{\lambda_0}$. For split primes in the extension
  $L/F$, the trace of $\rho_{\lambda_0}(\Frob_{\id{p}})$ equals
  $-\chi_{\lambda_0}(\id{p}_1)-\chi_{\lambda_0}(\id{p}_2)$, where
  $\id{p}\Om_L = \id{p}_1 \id{p}_2$ and for inert primes, the trace is
  zero. In particular, the same happens to
  $\rho_{\lambda}(\Frob_{\id{p}})$, for all $\lambda$ so
  $\rho_{\lambda}$ and $\Ind_{G_L}^{G_F} \chi_{\lambda}$ have the same
  trace and are thus isomorphic.
\end{proof}

To apply most MLT's theorems we will need to have some control of
the image of our residual Galois representations. In particular we
will need the image of its restriction to a cyclotomic extension to
be adequate. To avoid checking this particular condition at each
step of our chain of congruences, we will move to families with
``big image'', in the sense that for all but finitely many primes
$p$ (and in most steps, for all primes of bounded size), the
residual representation has an image containing $\SL_2(\F_p)$.

\begin{prop}
  Let ${\rho_\lambda}$ be a strictly compatible system of odd dihedral
  representations. Then there exists a strictly compatible system of odd
  non-dihedral representations $\{\rho_{2,\lambda}\}$ and a prime
  $\id{p}$ such that $\rho_{\id{p}} \equiv \rho_{2,\id{p}} \pmod{p}$
  and MLT holds in both directions.
\label{prop:non-dihedral}
\end{prop}

\begin{proof}
  It is well known that dihedral compatible families do not have
  Steinberg primes in the level, so we just need to add a Steinberg prime
  to our representation by some raising the level argument. Let
  $\lambda$ be a prime over a prime $p >5$ and such that:
  \begin{itemize}
  \item $p$ splits completely in $F$.
  \item $\lambda \not \in S$, i.e. $\rho_{\id{q}}$ is unramified at
    $\lambda$ if $\id{q} \neq \lambda$.
  \item $L$ and $F(\xi_p)$ are disjoint, i.e. $F(\xi_p) \cap L =F$.
  \end{itemize}
  We want to add a Steinberg prime $\id{q}$ modulo $\lambda$, and by
  the previous choice, we are in the hypothesis of Theorem 7.2.1 of
  \cite{Snowden} which says that it is enough to raise the level
  locally. The local problem is standard, and can be achieve by
  Tchebotarev's Theorem as follows: take $\id{q}$ inert in the
  extension $L/F$, so that $a_q=0$. Then the local raising the level
  condition becomes $ \normid{q} \equiv -1 \pmod p$, so we chose any
  such prime and get a global representation with the desired
  properties. The existence of a strictly compatible system attached to
  such global representation follows from Taylor's potential
  modularity result (see \cite{BLGGT2}) plus the argument from
  \cite{LuisFamilies}, as generalized in \cite{BLGGT2}. That MLT holds in both directions comes from
  Theorem~\ref{thm:MLT1}.
\end{proof}


\begin{prop}
  Let $\{\rho_\lambda\}$ be a strictly compatible system of odd,
  non-dihedral representations. Then there exists an integer $B$, such that if $\norm(\lambda)>B$ then $\SL_2(\F_p)\subset \im(\overline{\rho}_\lambda)$.
\label{prop:bigimage}
\end{prop}

\begin{proof}
  According to Dickson's classifications of subgroups of
  $\PGL_2(\FF_\lambda)$, when we consider the residual
  representations, they might:
  \begin{enumerate}
  \item contain $\PSL_2(\FF_\lambda)$,
  \item be reducible,
  \item be dihedral,
  \item be isomorphic to $A_4$, $S_4$ or $A_5$.
  \end{enumerate}
  We want to prove under the running hypothesis, there are only finitely
  many primes where the first case does not hold. But the second case
  is exactly Lemma $5.4$ of \cite{Calegari-Gee}, the third case
  Corollary $5.2$ of \cite{Calegari-Gee} (here we use the assumption
  that the system is not dihedral), and the last case is Lemma $5.3$
  of \cite{Calegari-Gee}.
\end{proof}

\begin{remark} Another way to prove the last Proposition (following
  the classical approach of Ribet) it to first apply some potential
  automorphy result (like in \cite{BLGGT2}) to deduce that the system
  is potentially modular. Then its restriction is isomorphic to that
  of a Hilbert modular form. Furthermore, since our abstract
  representations are not dihedral, the respective Hilbert modular
  form has no CM. But for Hilbert modular forms, such result is proven
  in \cite{Dimitrov} Proposition 0.1.
\end{remark}

By the above considerations, from now on we will only consider
non-dihedral families, so we will skip writing this hypothesis in the
next results.

\begin{prop}
Let $\{\rho_{1,\lambda}\}$ be a strictly compatible system of odd
irreducible Galois representations. Then there exists a compatible
system $\{\rho_{2,\lambda}\}$ of parallel weight $2$ representations
and a prime $\id{p}$ such that $\rho_{1,p} \equiv \rho_{2,p}
\pmod{p}$ and MLT holds in both directions.
\label{prop:parallel2}
\end{prop}

\begin{proof}
  Let $p>5$ be a prime which splits completely in $F$, does not divide
  the level of $\rho$, is larger than all weights of the system and
  such that the image of $\overline{\rho}_{1,\id{p}}$ is large.  Let
  $\det(\overline{\rho}) =\overline{\psi} \overline{\chi_p}$, where
  $\overline{\chi_p}$ is the reduction of the cyclotomic character,
  and let $\psi$ be any lift of $\overline{\psi}$. Then Theorem 7.6.1
  of \cite{Snowden} implies that $\overline{\rho}$ admits a weight two
  lift to $\overline{\QQ}_p$ which ramifies at the same primes as
  $\rho$ and $\psi$, with determinant $\psi \chi_p$. MLT hold in both
  directions by the same proof as the previous Proposition~\ref{prop:non-dihedral}.
\end{proof}




\subsection{Adding a good dihedral prime}

As already mentioned, while working with MLT one needs to ensure
that residual images are big enough to be in the hypothesis of such
a theorem. Sometimes, one needs the restriction of the residual
representation to the cyclotomic extension of $p$-th roots of unity
to  have adequate image, but for most MLT requiring irreducibility
of this restriction is enough (and for $p >5$ it is known that both
properties are equivalent). A way to get this property guaranteed at
most steps is by introducing to the level an extra prime which
forces the image modulo all primes up to a certain bound to be
``non-exceptional'', i.e. it is irreducible, and its
projectivization is not dihedral, nor the exceptional groups $A_4$,
$S_4$, $A_5$.  A way to get this, is by adding a ``good dihedral
prime" (with respect to the given bound) as was introduced by Khare
and Wintenberger in their work on Serre's conjecture. 

The difference with the classical setting is that since we work with
two strictly compatible systems at the same time, we need to add the
same good dihedral prime to both of them.

\begin{prop}
  Let $\{\rho_{1,\lambda}\}$ and $\{\rho_{2,\lambda}\}$ be two
  strictly compatible systems of continuous, odd, irreducible
  representations of parallel weight $2$.  Fix $B$ a positive integer,
  larger than $5$ and than all primes in the conductors of both
  systems. Let $p \equiv 1 \pmod 4$ be a rational prime such that:
  \begin{itemize}
  \item $p$ is bigger than $B$.
  \item $p$ splits completely in the compositum of the coefficient fields of $\rho_1$ and $\rho_2$.
  \item $p$ is relatively prime to the conductors of both systems.
  \item $\im (\overline{\rho}_{i,\id{p}})=\GL_2(\F_p)$, $i=1,2$, for some prime $\id{p}$ over $p$.
  \end{itemize}
  Then there exists a prime $q$ not dividing the conductor of the
  systems such that:

\begin{itemize}
\item $q \equiv -1 \pmod{p}$.
\item $q$ splits completely in the extension $F'$ given by the
  compositum of all quadratic extensions of $F$ ramified only at
  primes above rational primes $\ell < B$.
\item $q \equiv 1 \pmod{8}$.
\item There exists a prime ideal $\id{q}$ in $F$ over $q$ such that
  the image of $\overline{\rho}_{1,\id{p}}(\Frob_{\id{q}})$ and $\overline{\rho}_{2,\id{p}}(\Frob_{\id{q}})$ both have eigenvalues
  $1$ and $-1$.
\end{itemize}
With this choice of primes, there exists two strictly compatible
systems of continuous representations $\{\varrho_{1,\lambda}\}$ and
$\{\varrho_{2,\lambda}\}$ of parallel weight $2$ such that:
\begin{enumerate}
\item[$(1)$] $\overline{\rho}_{1,\id{p}} \simeq \overline{\varrho}_{1,\id{p}}$.
\item[$(2)$]$\overline{\rho}_{2,\id{p}} \simeq \overline{\varrho}_{2,\id{p}}$.
\item [$(3)$]$\{\varrho_{i,\lambda}\}$, for $i=1,2$, is locally good dihedral at $\id{q}$
(w.r.t. the bound $B$).
\item [$(4)$]$\varrho_{i,\id{p}}$, is Barsotti-Tate at all primes dividing
  $p$ and has the same type as $\rho_{i,\id{p}}$ locally at any prime
  other that $\id{q}$ for $i=1,2$.
\item [$(5)$] The congruences are MLT.
\end{enumerate}
\label{prop:gooddihedral}
\end{prop}

\begin{remark} We will not reproduce here the precise definition of
good-dihedral prime (it can be found in \cite{KW} and \cite{LuisBC})
or what it means for a compatible system to be locally good dihedral
at a prime w.r.t. certain bound $B$. Let us just recall that a
good-dihedral prime is a supercuspidal prime in a compatible system
such that the ramification at this prime forces the system to have
residual images containing $\SL_2(\F_p)$ every time the residual
characteristic $p$ is bounded by $B$ (and a bit more for $p=2$ or
$3$: also in these characteristics the image is forced to be
non-solvable).
\end{remark}

\begin{remark}
  We do not want to make precise what condition $(4)$
  of the last statement means (the local type), since it will not be
  important for our purposes, but what we prove is the following: for
  abstract representations, we will use Theorem 7.2.1 of
  \cite{Snowden}, where a ``type'' is understood as a Weil type (no
  information on the monodromy), while for modular representations
  (actually residually modular ones), we will use Theorem 3.2.2 of
  \cite{BD}, which uses the complete notion of type.
\end{remark}

\begin{proof}
  The existence of the primes $\id{p}$ and $\id{q}$ follows with
  almost the same arguments as \cite{LuisBC} (Lemma 3.3). The only
  difference is that we need to consider the compositum of $\Q(i), F$
  and the coefficient field of $\rho_1$ and $\rho_2$. Take $p$
  big enough (for the images to be large) and split in such
  extension. Then $q$ is chosen using Tchebotarev density Theorem, with
  the condition that it hits complex conjugation in the same suitable
  field.

  We are lead to prove the existence of a lift of
  $\bar{\rho}_{1,\id{p}}$ with the desired properties. By Theorem
  7.2.1 of \cite{Snowden}, or by Theorem 3.2.2 of \cite{BD} (see the
  last remark), we know that a global representation with the desired
  properties exists if and only if  locally the corresponding lifts do exist, so we only need
  to show which are the local deformation conditions:
  \begin{itemize}
  \item At the primes $\id{l} \neq \id{q}$, that of $\rho_i|_{G_{\id{l}}}$.
  \item $\varrho_i|_{D_{\id{q}}} =
    \Ind_{\QQ_q}^{\QQ_q[\sqrt{\epsilon}]}(\chi)$, where
    $\QQ_q[\sqrt{\epsilon}]$ is the unique quadratic unramified
    extension of $\QQ_q$, and $\chi$ is a character with order $p$
    (this is called type $C$ in \cite{Snowden}).
  \end{itemize}
  This proves the existence of $\varrho_{i,\id{p}}$. The congruence is
  MLT in both directions because of Theorem~\ref{thm:MLT1} (in this
  case both forms are of parallel Hodge-Tate weight $2$ which is
  smaller than the prime $p$).  Since we are in the hypothesis of an
  MLT theorem, we can put such representation into a strictly
  compatible system and get the result.
\end{proof}

\begin{remark}
  Although we stated the result for representations of parallel weight
  $2$, in general we can first move to parallel weight $2$ (using
  Proposition~\ref{prop:parallel2}) and then add the good dihedral
  prime.
\end{remark}

With this result, we can prove Mazur's Theorem.

\begin{thmA}[Mazur]
Let $F$ be a totally real number field, $f_0$ and $f_\infty$ two Hilbert
modular forms, whose weights are congruent modulo $2$, then there
exists Hilbert modular forms $\{h_i\}_{i=1}^n$ such that $h_1 =
f_0$, $h_n = f_\infty$ and
such that
for $i=1,\dots,n-1$ there exist $\lambda_i$ prime such that $\rho_{h_i,\lambda_i} \equiv \rho_{h_{i+1},\lambda_i} \pmod{\lambda_i}$.
\end{thmA}

\begin{proof}
  By the previous results, we can assume that both forms are of
  parallel weight $2$, and do not have complex multiplication (or we
  move to such a situation using Proposition~\ref{prop:non-dihedral}
  and Proposition~\ref{prop:parallel2}), for some common level
  $\Gamma_1(\id{n})$ (of course they might not be new at the same
  level). The idea now is to find both modular forms in the cohomology
  of a Shimura curve, and for this purpose we need to add an auxiliary
  prime to the level where both forms are not principal series (in
  case $[F:\QQ]$ is odd). So what we do is to raise the level of both
  forms at an auxiliary prime $\id{p}$ as in
  Proposition~\ref{prop:gooddihedral} (we could also add a Steinberg
  prime). Now we consider the Shimura curve $X^{\id{p}}(\id{n})$
  ramified at all infinite places of $F$ but one, and at the auxiliary
  prime $\id{p}$ if needed (depending whether $[F:\QQ]$ is odd or
  even) and with level $\id{n}$.

  We apply Mazur's argument just as in \cite{Mazur} (proposition 10.6,
  page 98) with some minor adjustments. The main idea of his proof is
  the following: if the Jacobian $J$ is a product of two abelian
  varieties $A \times B$, since $J$ decomposes (up to isogeny) as a
  product of simple factors with multiplicity one, there are no
  nontrivial homomorphisms from $A$ to $\hat{B}$ nor from $B$ to
  $\hat{A}$. Then the principal polarization of $J$ induces principal
  polarizations in $A$ and $B$, but a Jacobian cannot decompose as a
  nontrivial direct product of principally polarized abelian varieties
  (which follows from the irreducibility of its $\theta$-divisor, see
  \cite{Arba}). From this it follows that $\Spec \TT$ is connected,
  where $\TT$ denotes the Hecke algebra acting on $J$.

  In his original article Mazur was dealing with the curve $X_0(N)$,
  with $N$ a prime number, so there are no old forms appearing in
  $J_0(N)$. To use the same argument in our context, we have to deal
  with old forms as well, and the problem is that the abelian
  varieties $A_f$ corresponding to old forms do not appear with
  multiplicity one in the decomposition (up to isogeny) of the
  Jacobian of a modular or a Shimura curve. But this is not a problem if we
  observe that for what we want it is not necessary to prove the
  connectedness of $\Spec \TT$, it is enough to show that the anemic Hecke algebra $\TT_0$ generated only by the
  Hecke operators with index prime to the
  level is connected.
  Therefore, what we need is to discard
  the cases where the Jacobian of $X^{\id{p}}(\id{n})$ decomposes as a
  product of abelian varieties $A \times B$ with every simple factor in $A$ and every simple factor in $B$ being
  orthogonal (recall that now these simple factors need not appear with multiplicity one).
  In such case, the same proof as in Mazur's article applies,
  and gives the connectedness we are looking for.
\end{proof}

\section{Killing the level}\label{sec:killlevel}

Now that we have a good dihedral prime, which controls the image of
the residual representations in the families of the two
representations for small primes, we want to connect them at a chosen
level. From know on, we will only consider modular representations,
pointing out in each case the needed result for abstract ones. To
lower the level we will need the following result, whose proof is due
to Fred Diamond.

\begin{lemma}
  Let $\rho_f:\Gal(\overline{F}/F) \to \GL_2(K_{\id{p}})$ be a modular representation
  of level dividing $\id{n} p^r$ (where $\id{p}\mid p$ and $\id{p}\nmid \id{n}$) which
  satisfies the following hypothesis:
  \begin{itemize}
  \item The form $f$ has a supercuspidal prime $\id{q} \nmid p$.
  \item The residual image $\overline{\rho_f}$ is absolutely
    irreducible and not bad dihedral.
  \item The determinant $\det(\overline{\rho_f}|_{I_v}) = \chi_p^n$ for
    all $v|p$ and some $n$ independent of $v$.
  \end{itemize}
  Then $\overline{\rho}$ has a modular lift of level dividing $\id{n}$. Moreover
  we can assume the weight of the form is parallel and the character
  has order prime to $p$.
\label{lemma:Diamond}
\end{lemma}

\begin{proof}
  We can assume the weight of the form is parallel and the character
  has order prime to $p$.  We can also assume that $p$ is unramified
  in $F$ so we can appeal directly to \cite{BDJ} instead of slight,
  straightforward generalizations to the ramified case. We will follow
  the notation of the aforementioned article.

  By Proposition 2.10 and Corollary 2.11 of \cite{BDJ},
  $\overline{\rho}$ is modular of some weight $\sigma =
  \otimes_{v|p}\sigma_v$ and central character
  $\otimes_{v|p}(\N_{k_v/F_p})^{n-1}$.  The results of \cite{R} show
  that $\sigma$ is a Jordan-Holder constituent of the reduction of
  $\otimes_{\tau} \Symm^{k-1} \Om_E^2$ for some sufficiently large
  $k$, where the tensor product is over embeddings $\tau: F \to E$ for
  a sufficiently large number field $E$, viewed as contained both in $\CC$
  and in  $\overline{\QQ_p}$.  Another application of Proposition
  2.10 of \cite{BDJ} shows that $\overline{\rho}$ is modular of
  parallel weight $k$ and level prime to $p$.  Moreover the presence
  of a good dihedral prime $\id{q}$ allows us to use an indefinite
  quaternion algebra of discriminant dividing $\id{q}$ and hence
  assume the open compact subgroup has level dividing $\id{n}$ in the
  first application of Proposition 2.10 of \cite{BDJ}.  Moreover
  $\overline{\rho}$ is not badly dihedral in the sense of Lemma 4.11
  of \cite{BDJ}, so the arguments there ensure that we can assume the
  divisibility of $\id{n}$ by the level is preserved in the second
  application of Proposition 2.10 as well.  We can similarly ensure
  the conclusion on the central character.  For further details and
  more general results along these lines, see forthcoming work of
  Diamond and Reduzzi.

\end{proof}

\subsection{Modifying the non-Steinberg primes}

From now on we work under the assumption that MLFMT
(Assertion~\ref{ass:MLT4}) is true.  We call an abstract strictly
compatible system of continuous, odd, irreducible Galois
representations \emph{modular} if there exists a Hilbert modular form,
whose attached Galois representations matches the abstract one.

\begin{thm}
  Let $\{\rho_{1,\lambda}\}$ be a strictly compatible system of
  continuous, odd, irreducible, Galois representations attached to a
  Hilbert newform $f$ over a totally real number field $F$ of parallel
  weight $2$, containing in its conductor a locally good dihedral
  prime $\id{q}$ (w.r.t. some sufficiently large bound $B$).  Then:
  \begin{itemize}
  \item There exists a strictly compatible system of continuous, odd,
    irreducible, parallel weight $2$ representations
    $\{\rho_{2,\lambda}\}$, which is semistable at all primes except
    the same good dihedral prime $\id{q}$ and such that the Steinberg
    ramified primes are bounded in norm by $B$.
\item There exists a chain of congruences of compatible systems
  linking $\{\rho_{1,\lambda}\}$ and $\{\rho_{2,\lambda}\}$ such that
  all congruences involved occur in residual characteristics bounded
  by $B$ and are MLT.
  \end{itemize}
  In particular the system $\{\rho_{2,\lambda}\}$ is also modular.
\end{thm}

\begin{proof}
  Let $\lambda$ be a prime which is supercuspidal or principal
  series. Let $p$ be a prime number dividing the order of the
  character
   corresponding to ramification
  at $\lambda$. We consider two different cases: if $p$ is relative
  prime to $\lambda$, we call it the \emph{tamely ramified case},
  while the other case we call it the \emph{wildly ramified case}.

\medskip

\noindent {\bf The tamely ramified case.} Let $\id{p}$ be a prime
ideal in $K$ (the coefficient field) dividing $p$, and consider the
residual mod $\id{p}$ representation. Then the $p$-part of the
ramification is lost, so we take a minimal lift with the same parallel
weight $2$ (it exists by Theorem 3.2.2 in \cite{BD}). Observe that at this
step we are not only modifying the ramification type at $\lambda$, but
also at any other prime supercuspidal or principal series in the
prime-to-$p$ part of the conductor with ramification given by a
character of order divisible by $p$. Then by
Lemma~\ref{lemma:typesatp}, we might be in a mixed situation of
potentially Barsotti-Tate and ordinary representations. Using Theorem
3.2.2 of \cite{BD} (recall that this is a special case of our Assumption
~\ref{ass:MLT4}), we get rid of this $p$-part of the inertia with an
MLT congruence. Iterating this process, we end up killing all tamely
ramified ramification given by characters, i.e. the prime-to-$p$ part
of the ramification at primes dividing $p$ is killed. So we are
reduced to the case where all primes in the conductor are either
Steinberg or with ramification given by a prime order character whose
order is divisible by the ramified prime.

\medskip

\noindent {\bf The wildly ramified case}. In this case, we will move
the wildly ramified primes (up to twist) to tamely ramified primes, so
the previous argument ends our proof. For a prime $\id{t}$ in the
conductor of the wildly ramified case, let us call $t$ the rational
prime below $\id{t}$ and consider a mod $t$ congruence, up to twist by
some finite order character $\psi$, with the Galois representation
corresponding to a Hilbert newform $H$ of parallel weight $2$ with at
most $\id{t}$ to the first power in the level (i.e. $\Gamma_1$ at
$\id{t}$), and the same for the other primes dividing $t$.  The
existence of such a form comes from Lemma~\ref{lemma:Diamond}.

By level-lowering, we can assume that the only extra primes in the
level of $H$ are those primes that have been introduced to the
residual conductor while twisting by a character $\psi$. It is easy to
see that such character can be chosen such that at primes other than
those dividing $t$ it has square-free conductor. To this congruence,
Theorem 3.2.2 of \cite{BD} applies (the conditions for this theorem are
preserved by twisting, and modularity too) so we are reduced to a case
where we have a system that is either Steinberg or tamely ramified
principal series at primes dividing $t$, and tamely ramified principal
series at all extra primes introduced by $\psi$. If we iterate this
process at all wildly ramified primes, we end up with a system with no
wildly ramified primes. We repeat the previous case procedure of killing all
ramification given by tamely ramified characters, but now in the
absence of wild ramification we finish with a compatible system such
that all its ramified primes other than the good-dihedral prime are
Steinberg. It is not hard to see that in all this process, for a
suitably chosen bound $B$, all auxiliary primes can be taken to be
smaller than $B$.
\end{proof}

\begin{remark}
  Except for the application of Corollary 2.12 of \cite{BDJ} at a key
  point, and a better control of the local types (which seems
  reasonable for abstract representations), all congruences in the
  above proof are known to exist for abstract Galois representation,
  so an analogue of the above result for abstract compatible systems
  can be proved assuming that this result from \cite{BDJ} generalizes
  to the abstract setting (thus relating the modularity of any
  geometric compatible system to that of a system with only Steinberg
  primes).
\end{remark}



\subsection{Killing the Steinberg primes}

This part of the process is a little more delicate, and the MLT
theorems are more restrictive, so what we do first is move the
Steinberg primes to primes which split completely in $F$.

\begin{thm}
  Suppose that Assumption ~\ref{ass:MLT4} is true. Let
  $\{\rho_{1,\lambda}\}$ be a system of modular continuous, odd,
  irreducible, parallel weight $2$ representations with a big locally
  good dihedral prime $\id{q}$.  Let $\mathcal{L} \mid \ell$ ($\ell \in
  \ZZ$) be a Steinberg prime of the system which does not split
  completely in $F$. We also assume that the system is either unramified or Steinberg at all primes dividing $\ell$.
  Then:
  \begin{itemize}
  \item there exists a strictly compatible system
    $\{\rho_{2,\lambda}\}$ of continuous, odd, irreducible, parallel
    weight $2$ representations, which has the same ramification
    behavior at all primes except those dividing $\ell$, where it is
    unramified, and has at most a set of extra Steinberg primes, all
    of them dividing the same rational prime which splits completely
    in $F$.
  \item there exists a chain of MLT congruences linking the two systems.
  \end{itemize}
  In particular, the system $\{\rho_{2,\lambda}\}$ is also modular.
\label{thm:movingsteinbergprimes}
\end{thm}

\begin{proof}
  We look at $\rho_{1,\lambda'}$ for a prime $\lambda'$ dividing
  $\ell$ and we reduce it modulo $\ell$.  We can construct then a
  modular lift which is unramified at $\mathcal{L}$ and at all primes
  dividing $\ell$, and with weights among those predicted by the
  Serre's weights of the residual representation.  This follows from
  iterated applications of Lemma~\ref{lemma:Diamond}
  since the determinant locally at primes above $\ell$ is a fixed
  power of the cyclotomic character, and the form (without $\ell$ in
  the level) has parallel weight as well.  Note that this congruence
  is MLT, at least under Assumption ~\ref{ass:MLT4}, which applies
  because at Steinberg primes, weight $2$ representations are
  ordinary, and the crystalline lift of higher weight can also be
  taken to be ordinary (observe that to deduce modularity of the
  weight $2$ family from the other one we can apply
  Theorem 3.2.2 in \cite{BD}). \\
  Let $p$ be a big prime (bigger than the weights of this second
  family) which splits completely in $F$, then by \cite{Snowden} we
  can construct a third family with parallel weight $2$ by looking
  modulo $p$, and this congruence is MLT because of the results in
  \cite{BLGGT2} (we are comparing a Fontaine-Laffaille with a
  potentially Barsotti-Tate representation, so they are both
  potentially diagonalizable). We apply the previous section method to
  this new family to end up with a representation which is at most
  Steinberg at all primes dividing $p$ as desired. Observe that in
  this last step no extra ramified primes are introduced because we do
  not have wild ramification at $p$ (the parallel weight $2$ lift
  implies tamely ramified principal series at all primes above $p$).
\end{proof}

Now that we have only totally split Steinberg primes in our family, we
can get rid of the Steinberg primes.

\begin{thm}
  Let $\{\rho_{1,\lambda}\}$ be a system of modular, irreducible,
  parallel weight $2$ representations, whose ramification consists of
  a big locally good dihedral prime $\id{q}$ and Steinberg primes
  which split completely in $F$. Then:
  \begin{itemize}
  \item There exists a strictly compatible system
    $\{\rho_{2,\lambda}\}$ of continuous, odd, irreducible
    representations which are only ramified at the locally good
    dihedral prime $\id{q}$.
  \item There exists a chain of MLT congruences linking the two systems.
  \end{itemize}
  In particular, the system $\{\rho_{2,\lambda}\}$ is also modular.
\label{thm:removingsteinbergprimes}
\end{thm}

\begin{proof}
  The procedure is quite similar to the one in the previous Theorem,
  but now we can use Theorem~\ref{thm:MLT1}. For each Steinberg prime
  of residual characteristic $\ell$ we look at
  $\overline{\rho}_{1,\lambda}$ with $\lambda \mid \ell$. There exists
  a minimal lift which is crystalline at all primes of residual
  characteristic $\ell$ by Lemma~\ref{lemma:Diamond}. Now, since the
  prime is split, the congruence is MLT by the cited Theorem. The only
  delicate point here is that to apply Theorem~\ref{thm:MLT1} we need
  the residual characteristic to be different from $2$ and $3$, so if
  we have such a small Steinberg prime, we first apply
  Theorem~\ref{thm:movingsteinbergprimes} to transfer the ramification
  to Steinberg ramification at some larger split prime.
\end{proof}

Remark: In the previous two Theorems, modularity of the given system
was only used to ensure existence of the lift corresponding to a
system with no $\ell$ in its conductor, and this was deduced from
results of \cite{BDJ}. In the case of abstract Galois representations,
after computing the Serre's weights of the residual representation
one should be able to propose locally at all primes above $\ell$ a
crystalline lift of the residual local representation, but then the
problem is that in order to apply results such as those in \cite{BLGGT2}
that guarantee existence of a global lift with such local
conditions, one should be in a potentially diagonalizable case. In particular, under
the conjecture that all potentially crystalline representations are potentially
diagonalizable, the previous two theorems shall generalize to the abstract
setting (of course, the conclusion that the last system is modular
shall also be removed).

\begin{proof}[Proof of Theorem~\ref{thm:main}] With the machinery
  developed in the previous chapters, starting with two modular
  representations, we can take each of them to representations which
  are only ramified at the good dihedral prime $\id{q}$. The only
  problem is that while killing the Steinberg primes, we lost control
  over the weights, so to take the families to parallel weight $2$, we
  chose an auxiliary prime $p$ which splits completely in $F$ and
  consider a parallel weight $2$ lift of each family modulo $p$, with
  the same local type at $\id{q}$. Note that we cannot assure that the
  forms we get at the end of the process will be newforms for
  $\Gamma_0(p\id{q}^2)$, because at some of the prime ideals of $F$
  dividing $p$ our representations could be unramified.
\end{proof}

\section{Further developments}\label{sec:last}

As mentioned before, although we have some good control on the level
of the forms we started with, the primes in the level are not
explicit. One can go further and change the primes in the level for
smaller and concrete ones, so as to check whether the connectedness
of the Hecke algebras in Mazur's Theorem corresponds to a chain
where all congruences are MLT or not. For this purpose, one can use
the notion of \emph{micro good dihedral
  primes}. Adding a micro good dihedral prime (which is chosen asking
some splitting behavior in the base field) one can get rid of the
good dihedral prime, and also bound the Steinberg primes in the
level. These ideas, although standard (see for example
\cite{LuisBC}) are more delicate, and involve some technicalities
that we prefer to avoid in the present article. With this control,
one can give an algorithm that given a totally real number field,
checks whether our approach implies Base Change over that field via
a finite computation involving Hilbert modular forms. See
\cite{LuisAriel} for more details. \\
Concerning Serre's conjecture over a specified small real quadratic
field, one should carefully check that the chain of congruences
constructed carries on to the case of abstract representations, and
then once having reduced the problem to the case of representations
of concrete small invariants (those where the above process
concludes, after the introduction of the micro good dihedral prime)
one should connect such representations to some ``base case" where
modularity or residual reducibility is known (applying for example
the result of Schoof recalled in the introduction), checking on the
way that all congruences are MLT (with the advantage that over
quadratic fields there are also MLT of Skinner and Wiles that deal
with the residually reducible case, under suitable assumptions). We
plan to check in a future work if this strategy succeeds in giving a
proof of Serre's conjecture over some small real quadratic field.





\bibliographystyle{alpha}
\bibliography{biblio}

\begin{thebibliography}{ACGH85}

\bibitem[ACGH85]{Arba}
E.~Arbarello, M.~Cornalba, P.~A. Griffiths, and J.~Harris.
\newblock {\em Geometry of algebraic curves. {V}ol. {I}}, volume 267 of {\em
  Grundlehren der Mathematischen Wissenschaften [Fundamental Principles of
  Mathematical Sciences]}.
\newblock Springer-Verlag, New York, 1985.

\bibitem[BD]{BD}
Christophe Breuil and Fred Diamond.
\newblock Formes modulaires de {H}ilbert modulo p et valeurs d'extensions entre
  caract\`eres galoisiens.
\newblock {\em Ann. Scient. de l'E.N.S., to appear}.

\bibitem[BDJ10]{BDJ}
Kevin Buzzard, Fred Diamond, and Frazer Jarvis.
\newblock On {S}erre's conjecture for mod {$\ell$} {G}alois representations
  over totally real fields.
\newblock {\em Duke Math. J.}, 155(1):105--161, 2010.

\bibitem[BLGGT]{BLGGT2}
Thomas Barnet-Lamb, Toby Gee, David Geraghty, and Richard Taylor.
\newblock Potential automorphy and change of weight.
\newblock {\em Annals of Math., to appear}.

\bibitem[CG11]{Calegari-Gee}
Frank Calegari and Toby Gee.
\newblock Irreducibility of automorphic galois representations of {${\rm
  GL}(n)$}, $n$ at most $5$.
\newblock {\em arXiv:1104.4827 [math.NT]}, 2011.

\bibitem[DG12]{DiGe}
Luis Dieulefait and Toby Gee.
\newblock Automorphy lifting for small $\ell$ - appendix b to "automorphy of
  $symm^ 5(gl(2))$ and base change".
\newblock {\em arXiv:1209.5105 [math.NT]}, 2012.

\bibitem[Die04]{LuisFamilies}
Luis~V. Dieulefait.
\newblock Existence of families of {G}alois representations and new cases of
  the {F}ontaine-{M}azur conjecture.
\newblock {\em J. Reine Angew. Math.}, 577:147--151, 2004.

\bibitem[Die12a]{LuisBC2}
Luis Dieulefait.
\newblock Automorphy of {$Symm^5({\rm GL}(2))$} and base change.
\newblock {\em arXiv:1208.3946[mathNT]}, 2012.

\bibitem[Die12b]{LuisBC}
Luis Dieulefait.
\newblock Langlands base change for {${\rm GL}(2)$}.
\newblock {\em Ann. of Math. (2)}, 176(2):1015--1038, 2012.

\bibitem[Dim05]{Dimitrov}
Mladen Dimitrov.
\newblock Galois representations modulo {$p$} and cohomology of {H}ilbert
  modular varieties.
\newblock {\em Ann. Sci. \'Ecole Norm. Sup. (4)}, 38(4):505--551, 2005.

\bibitem[DP]{LuisAriel}
Luis Dieulefait and Ariel Pacetti.
\newblock Examples of base change for real quadratic fields.
\newblock {\em In preparation}.

\bibitem[Gur12]{Gur}
Robert Guralnick.
\newblock Adequacy of representations of finite groups of lie type - appendix a
  to "automorphy of $symm^ 5(gl(2))$ and base change".
\newblock {\em arXiv:1208.4128 [math.NT]}, 2012.

\bibitem[HT13]{Hu-Tan}
Yongquan Hu and Fucheng Tan.
\newblock The {B}reuil-{M}ezard conjecture for non-scalar split residual
  representations.
\newblock {\em arXiv:1309.1658 [math.NT]}, 2013.

\bibitem[KK03]{KK}
Chandrashekhar Khare and Ian Kiming.
\newblock Mod {$pq$} {G}alois representations and {S}erre's conjecture.
\newblock {\em J. Number Theory}, 98(2):329--347, 2003.

\bibitem[KW09]{KW}
Chandrashekhar Khare and Jean-Pierre Wintenberger.
\newblock Serre's modularity conjecture. {I}.
\newblock {\em Invent. Math.}, 178(3):485--504, 2009.

\bibitem[Maz77]{Mazur}
B.~Mazur.
\newblock Modular curves and the {E}isenstein ideal.
\newblock {\em Inst. Hautes \'Etudes Sci. Publ. Math.}, (47):33--186 (1978),
  1977.

\bibitem[Roz12]{R}
Sandra Rozensztajn.
\newblock Asymptotic values of modular multiplicities for {${\rm GL}_2$}.
\newblock {\em arXiv:1209.5666, to appear in J. Th\'eorie des Nombres
  Bordeaux}, 2012.

\bibitem[Sch12]{Schoof}
Ren{\'e} Schoof.
\newblock Semistable abelian varieties with good reduction outside 15.
\newblock {\em Manuscripta Math.}, 139(1-2):49--70, 2012.

\bibitem[Ser98]{serre-book}
Jean-Pierre Serre.
\newblock {\em Abelian {$l$}-adic representations and elliptic curves},
  volume~7 of {\em Research Notes in Mathematics}.
\newblock A K Peters Ltd., Wellesley, MA, 1998.
\newblock With the collaboration of Willem Kuyk and John Labute, Revised
  reprint of the 1968 original.

\bibitem[Sno09]{Snowden}
Andrew Snowden.
\newblock On two dimensional weight two odd representations of totally real
  fields.
\newblock {\em arXiv:0905.4266v1 [math.NT]}, 2009.

\bibitem[SW99]{SW}
C.~M. Skinner and A.~J. Wiles.
\newblock Residually reducible representations and modular forms.
\newblock {\em Inst. Hautes \'Etudes Sci. Publ. Math.}, (89):5--126 (2000),
  1999.

\end{thebibliography}
\end{document}